\documentclass[12pt,reqno,tbtags]{article}

\usepackage{amsthm}
\usepackage{titlesec}
\usepackage{tikz}
\usetikzlibrary{shapes,backgrounds,calc}
\usepackage{float}

\titlelabel{\thetitle.\quad}
\usepackage{amssymb}
\usepackage{amsfonts, mathrsfs}
\usepackage{amsmath}
\usepackage[numbers, sort&compress]{natbib}
\usepackage{graphicx}
\usepackage{vmargin, enumerate}
\usepackage{authblk}
\usepackage{setspace}
\usepackage{paralist}
\usepackage[pdftex]{hyperref}
 
\usepackage{mathtools}
\usepackage{comment}

\newcommand{\E}[1]{{\mathbb E}\left[#1\right]}

\newcommand{\pl}[1]{\text{#1}}
\newcommand{\eps}{\varepsilon}
\newcommand{\mc}[1]{\mathcal{#1}}

\newtheorem{thm}{Theorem}[section]
\newtheorem{lem}[thm]{Lemma}

\newtheorem{cor}[thm]{Corollary}

\newtheorem{definition}[thm]{Definition}
\newtheorem{conjecture}{Conjecture}[section]

\newtheorem{fact}[thm]{Fact}
\newtheorem{claim}[thm]{Claim}

\newcommand\udist{1.5cm}
\newcommand\vdist{4.0cm}
\newcommand\wdist{5.7cm}
\newcommand\shift{-8.5}

\hypersetup{
    bookmarks=true,         
    unicode=false,          
    pdftoolbar=true,        
    pdfmenubar=true,        
    pdffitwindow=true,      
    pdftitle={My title},    
    pdfauthor={Author},     
    pdfsubject={Subject},   
    pdfnewwindow=true,      
    pdfkeywords={keywords}, 
    colorlinks=true,       
    linkcolor=blue,          
    citecolor=blue,        
    filecolor=blue,      
    urlcolor=blue           
}
\makeatletter
\tikzset{circle split part fill/.style  args={#1,#2}{%
 alias=tmp@name, 
  postaction={%
    insert path={
     \pgfextra{%
     \pgfpointdiff{\pgfpointanchor{\pgf@node@name}{center}}%
                  {\pgfpointanchor{\pgf@node@name}{east}}%
     \pgfmathsetmacro\insiderad{\pgf@x}
      \fill[#1] (\pgf@node@name.base) ([xshift=-\pgflinewidth]\pgf@node@name.east) arc
                          (0:180:\insiderad-\pgflinewidth)--cycle;
      \fill[#2] (\pgf@node@name.base) ([xshift=\pgflinewidth]\pgf@node@name.west)  arc
                           (180:360:\insiderad-\pgflinewidth)--cycle;            
         }}}}}
 \makeatother

\title{Sparse halves in dense triangle-free graphs}
\author[1]{Sergey Norin\thanks{snorine@gmail.com}}
\author[2]{Liana Yepremyan\thanks{liana.yepremyan@mail.mcgill.ca}}
\affil[1]{\small{Department of Mathematics and Statistics, McGill University}}
\affil[2]{\small{School of Computer Science, McGill University}}

\date{}
\begin{document}

\maketitle

\begin{abstract}
Erd\H{o}s~\cite{Erdosfirst} conjectured that every  triangle-free graph \(G\) on \(n\) vertices contains a set of \(\lfloor n/2 \rfloor\)  vertices that spans at most \(n^2 /50\) edges. Krivelevich proved the conjecture for graphs with minimum degree at least \(\frac{2}{5}n\) \cite{krivelevich}. In \cite{Sudakov} Keevash and Sudakov improved this result to graphs with average degree at least \(\frac{2}{5}n\). We strengthen these results by showing that the conjecture holds for graphs with minimum degree  at least \( \frac{5}{14}n\) and for graphs with average degree at least \(\left(\frac{2}{5} - \gamma \right)n\) for some absolute \(\gamma >0\). Moreover, we show that  the conjecture is true for graphs which are close to the Petersen graph in edit distance.
\end{abstract}
\begin{keywords}
Triangle-free graph, sparse half, minimum degree, Petersen graph, edit distance, blowup
\end{keywords}

\section{Introduction}

In this paper we consider the edge distribution in triangle-free graphs. A fundamental result in extremal graph theory, Tur\'{a}n's theorem implies that if a graph $G$ on \(n\) vertices has more than $(1-1/r) {n \choose 2}$ edges, then G has $K_{r+1}$ as a subgraph.  One can consider the following generalization of this problem that was first studied by Erd\H{o}s, Faudree,  Rousseau and Schelp \cite{Erdos1} in 1990. Fix \(\alpha \in (0,1]\), and suppose that every set of \(\alpha n\) vertices in $G$ induces more than $\beta n^2$ edges. What is the smallest \(\beta:=\beta(\alpha,r)\) that forces $G$ to contain $K_{r+1}$?

In particular, one of Erd\H{o}s's old and favorite conjectures says that \(\beta(\frac{1}{2}, 2) = \frac{1}{50}\) \cite{Erdosfirst}; he first proposed this in 1975 and offered a \$250 prize for its solution later, in \cite{Erdos3}. It is easy to check that the bound $1/50$ is tight. It is achieved  on  the uniform \emph{blowup} of \(C_5\) which is obtained from the \(5\)-cycle by replacing
each vertex \(i\) by an independent set \(V_i\) of size \(n/5\) (for simplicity assume \(n\) is divisible by \(5\)) and each edge \(ij\) by a complete bipartite graph joining \(V_i\) and \(V_j\). The  blowup of the Petersen graph also achieves this bound tightly, see Figure~\ref{petersen}.

\begin{conjecture}[Erd\H{o}s]
\label{mainconjecture}
Any triangle-free graph \(G\) on \(n\) vertices contains a set of \(\lfloor n/2 \rfloor\) vertices that spans at most \(n^2/50\) edges.
\end{conjecture}

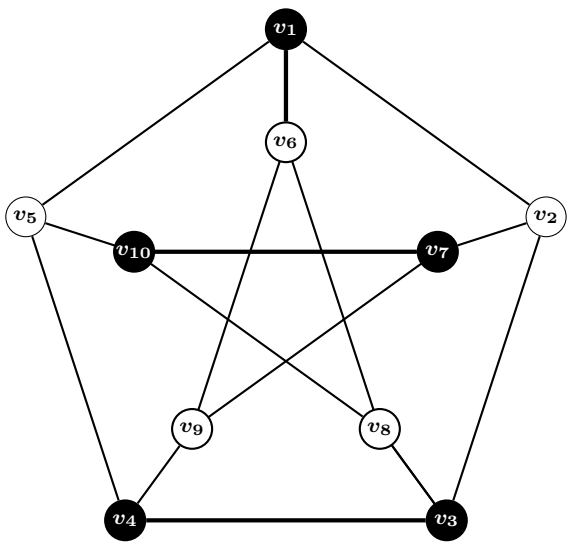
\begin{figure}[htbp]
\label{petersen}
\begin{center}
\begin{tikzpicture}[label distance=1mm]
    \tikzstyle{every node}=[draw,circle,fill=white,minimum size=15pt,
                            inner sep=0pt]
    \draw (0,0) node (v10)[draw=black,fill=black]{\textcolor{white}{\scriptsize$\boldsymbol{v_{10}}$}}
        [thick]-- ++(0:\vdist) node (v7)[draw=black,fill=black] {\textcolor{white}{\scriptsize$\boldsymbol{v_7}$}}
        -- ++(216:\vdist) node (v9){\textcolor{black}{\scriptsize$\boldsymbol{v_9}$}}
        -- ++(72:\vdist) node (v6) {\textcolor{black}{\scriptsize$\boldsymbol{v_6}$}}
        -- ++(288:\vdist) node (v8) {\textcolor{black}{\scriptsize$\boldsymbol{v_8}$}}
        -- ++(306:\udist) node (v3)[draw=black,fill=black]{\textcolor{white}{\scriptsize$\boldsymbol{v_3}$}};
	\path (v10) ++(162:\udist) node (v5){\textcolor{black}{\scriptsize$\boldsymbol{v_5}$}};
	\path (v7) ++(18:\udist) node (v2){\textcolor{black}{\scriptsize$\boldsymbol{v_2}$}};
	\path (v9) ++(234:\udist) node (v4)[draw=black,fill=black, thick]{\textcolor{white}{\scriptsize$\boldsymbol{v_4}$}};
	\path (v6) ++(90:\udist) node (v1)[draw=black,fill=black, thick]{\textcolor{white}{\scriptsize$\boldsymbol{v_1}$}};

      \draw [thick] (v10)--(v8);
\draw [thick] (v10)--(v5);
\draw  [thick](v7)--(v2);
\draw [thick] (v9)--(v4);
\draw  [thick](v6)--(v1);
\draw [ thick] (v8)--(v3);
\draw [ thick] (v5)--(v1);
\draw [ thick] (v2)--(v1);
\draw [ thick] (v2)--(v3);
\draw [ultra thick] (v4)--(v3);
\draw [ thick] (v4)--(v5);
\draw [ultra thick] (v1)--(v6);
\draw [ ultra thick] (v10)--(v7);

\end{tikzpicture}
\end{center}
\caption{A sparse half in the uniform blowup of Petersen graph.}
\end{figure}

In \cite{krivelevich} Krivelevich proved  that the conjecture holds if \(1/50\) is replaced by \(1/36\). He also showed that  it is true for triangle-free graphs  with minimum degree \(\frac{2}{5}n\). In Section~\ref{sec:minimum} we improve this result by proving the following theorem.

\begin{thm}
\label{maintheorem1}
Every  triangle-free graph on \(n\) vertices with minimum degree at least \(\frac{5}{14}n\) contains a set of \(\lfloor n/2 \rfloor \) vertices that spans at most \(n^2/50\) edges.
\end{thm}

Our proof of Theorem~\ref{maintheorem1} is mainly based on the structural characterization of the graphs with minimum degree at least \(\frac{5}{14}n\), established by Jin, Chen  and Koh in \cite{jinchenkoh,jin}. We also use some  averaging arguments similar to the ones used  in \cite{Sudakov, krivelevich}.

\vskip 5pt
Keevash and Sudakov ~\cite{Sudakov} improved Krivelevich's result of~\cite{krivelevich} showing that the conjecture holds for graphs with average degree \(\frac{2}{5}n\). In Section~\ref{sec:average} we extend their result as follows.

\begin{thm}
\label{maintheorem2}
There exists \(\gamma >0\) such that every triangle-free graph on \(n\) vertices with at least \((\frac{1}{5}-\gamma )n^2\) edges  contains a set of  \(\lfloor n/2 \rfloor \) vertices that spans at most \(n^2/50\) edges.
\end{thm}

Finally, we study the validity of the conjecture in the neighborhood of the known extremal examples, in the following sense. In the uniform blowup of the Petersen graph every set of \(\lfloor n/2 \rfloor \) has at least \(n^2/50\) edges (see Figure~\ref{petersen}). To the best of our knowledge, the uniform blowups of the Petersen graph and $C_5$, as defined in the beginning of the introduction, are the only known examples for which Conjecture~\ref{mainconjecture} is tight. In Section~\ref{s:uniform} we develop a set of tools which allow us to prove the Conjecture~\ref{mainconjecture} for the classes of graphs which are close to a fixed graph in edit distance. In Section~\ref{sec:petersen} we use these tools to verify the conjecture for graphs which are close to the  Petersen graph, while Theorem~\ref{maintheorem2} shows that it also holds for graphs close to the $5$-cycle. These results can be considered as a proof of
a local version of Conjecture~\ref{mainconjecture}, in the spirit of recent results of Lov\'asz~\cite{LovSidorenko} and Razborov~\cite{RazborovCH,RazborovTuran34}. Lov\'{a}sz~\cite{LovSidorenko} proved Sidorenko's conjecture locally in the neighborhood  of the conjectured extremal example with respect to the cut-norm (and consequently, with respect to the edit distance). Razborov proved that the Caccetta-H\"{a}ggkvist conjecture~\cite{RazborovCH} and the Tur\'{a}n's \((3,4)\)-problem~\cite{RazborovTuran34} hold locally in the neighborhood of families of the known examples in a slightly different sense, as he additionally forbade certain induced subconfigurations.

\section{Notation and Preliminary Results}

For a graph \(G\), we let \(\pl{v}(G):=|V(G)|\)  and \(\pl{e}(G):=|E(G)|\). Denote by \(N_{G}(v)\) the neighborhood of a vertex \(v \in V(G)\) and by \(d_G(v)\) the degree of \(v\) (i.e. \(d_G(v):=|N_G(v)|\)). Whenever there is no ambiguity, we will use \(N(v)\) and \(d(v)\) for shorter notation. The maximum and the minimum degrees are denoted by \(\Delta(G)\) and \(\delta(G)\), respectively.

A  graph \(G\) is called \emph{triangle-free} if it does not contain a triangle, it is \emph{maximal} triangle-free if adding any edge creates a triangle. We say that $G$ \emph{contains a sparse half} if there exists a set of \(\lfloor n/2 \rfloor\) vertices in $G$ that spans at most \(n^2/50\)  edges. Conjecture~\ref{mainconjecture} says that there must exist a sparse half in every triangle-free graph.

We say that $\omega: V(G) \to (0,1)$ is a \emph{weight function on} $G$ if $\sum_{v\in V(G)} {\omega(v)} =1.$
A pair \((G, \omega)\), where $\omega$ is a weight function on $G$ is called a \emph{weighted graph}.
The weight $\omega(e)$ of an edge \(e= (u,v)\) in $(G,\omega)$ is defined as \(\omega (u)\cdot \omega(v)\). For a set $X$ of vertices or edges in $G$ let $\omega(X):=\sum_{x \in X}\omega(x)$.
The \emph{degree} of a vertex \(v\) in a weighted graph \((G,\omega)\) is defined as \(\omega(N(v))\). The \emph{minimum degree} of the weighted graph \((G,\omega)\)  is denoted by \({\delta}(G,\omega)\).

We call a real function \(\pl{s}:V(G)\rightarrow \mathbb{R^+}\) \emph{a half} of \(G\) if   $\pl{s}(v) \leq \omega(v)$ for every \(v \in V(G)\), and $\pl{s}(V(G)) := \sum_{v \in V(G)}{\pl{s}(v)} =1/2.$ For every edge \(e=(u,v)\), we define \(\pl{s}(e):=\pl{s}(u)\pl{s}(v)\). Let \(\pl{s}\) be a half of \(G\), if $\pl{s}(E(G)):= \sum_{e\in E(G)}{\pl{s}(e)}\leq \frac{1}{50}$ then \(\pl{s}\) is called a \emph{sparse half}.

There is a natural way of associating a weighted graph  to a graph \(G\) of order \(n\); that is, we can assign to each vertex \(v\in V(G)\)  weight equal to \(\frac{1}{n}\).  For each graph \(G\), the corresponding weighted graph defined as above, is called \emph{uniformly weighted} \(G\); we denote it by \((G,{\omega}_u)\).

\begin{lem}
\label{auxlem1} If \((G,\omega_u)\) has a sparse half, then so does \(G\).
\end{lem}
\begin{proof}
We claim that, if \((G,\omega_u)\) has a sparse half, then it has a sparse half  \(\pl{s}\) such that \(\pl{s}(v)=0\) or \(\pl{s}(v)=\frac{1}{n}\) for all  \(v \in V(G)\) except for possibly one vertex.

Indeed, let us choose a sparse half \(\pl{s}\) of \(G\) such that the number of vertices \(v\in V(G)\) such that either \(\pl{s}(v)=0\) or \(\pl{s}(v)=\frac{1}{n}\) is maximum. We show that there exists  at most one vertex \(u\) such that  \(0<\pl{s}(u)<\frac{1}{n}\).

Suppose not and there exist \(u,v\in V(G)\) such that \(0<\pl{s}(u),\pl{s}(v) <\frac{1}{n}\). We define a new half \(\pl{s}':V(G)\rightarrow \mathbb{R}^+\).
Let \(\pl{s}'\) be the same as \(\pl{s}\) on all vertices of \(G\), except \(u\) and \(v\). Without loss of generality, suppose $\pl{s}(N(u)) \leq \pl{s}(N(v))$.
Let \(\delta = \min\{\pl{s}(u), \frac{1}{n}-\pl{s}(v)\}\) and define \(\pl{s}'(u) = \pl{s}(u)-\delta\) and \(\pl{s}'(v) = \pl{s}(v) + \delta\). We will show that $\pl{s}'$ is a sparse half.

Suppose that \(u\) and \(v\) are adjacent, then
\begin{align*}
\pl{s}'(E(G))  &= \pl{s}(E(G)) - \sum_{x\in N(u) \atop x\neq v}{\delta \pl{s}(x)} + \sum_{y\in N(v) \atop y\neq u }{\delta \pl{s}(y)} + \pl{s}'(u)\pl{s}'(v) - \pl{s}(u)\pl{s}(v) \\
&=\pl{s}(E(G)) - \sum_{x\in N(u) \atop x\neq v}{\delta \pl{s}(x)} + \sum_{y\in N(v) \atop y\neq u }{\delta \pl{s}(y)} -\delta \pl{s}(v) + \delta \pl{s}(u) -\delta^2  \\
&= \pl{s}(E(G)) - \delta \left(\pl{s}(N(u))-\pl{s}(N(v))\right) -\delta^2 < \pl{s}(E(G)).
\end{align*}
The calculation in the case when \(u\) and \(v\) are non-adjacent is similar.

It follows that $\pl{s}'$ contradicts the choice of \(\pl{s}\). Hence
for all vertices \(v\in V(G)\) except maybe one vertex either \(\pl{s}(v)=0\) or \(\pl{s}(v)=\frac{1}{n}\). Let \[S = \{v \in V(G)\:|\: \pl{s}(v)=1/n\}.\] It follows from the previous observation that $|S| \geq \lfloor n/2 \rfloor$.   It is easy to see that \(\pl{e}(G[S]) \leq n^2\pl{s}(E(G)) \leq n^2/50\), where \(G[S]\) is defined to be the subgraph of \(G\) induced by the vertex set \(S\). This shows that $S$ ia a sparse half in $G$, as desired.
\end{proof}

Lemma~\ref{auxlem1} allows us to work with weighted graphs, which proves to be convenient.  We prove that every weighted triangle-free graph  with minimum degree at least \(5/14\) contains a sparse half. Our proof uses structural characterization of these graphs found by Jin, Chen  and Koh in \cite{jinchenkoh,jin}.  To state their result we need a few additional definitions.

A mapping \(\varphi :V(G)\rightarrow V(H)\) for graphs $G$ and $H$ is  a \emph{homomorphism from $G$ to $H$}, if it is edge-preserving, that is, for any pair of adjacent vertices \(u, v \in V(G)\),  $\varphi(u)$ and $\varphi(v)$ are adjacent in $H$. We say that $G$ is \emph{of $H$-type} if there exists a homomorphism from $G$ to $H$.   Let \(\varphi\) be a surjective homomorphism and \(\omega\) be a weight function on \(G\). We define a weight function \(\omega_{\varphi}\) on \(H\) in the following way. For every vertex \(v\in V(H)\), let $\omega_{\varphi}(v) := \omega({\varphi}^{-1}(v))$. The next lemma shows that a sparse half in a  homomorphic image of a graph \(G\) can be lifted to a sparse half in the graph \(G\).

\begin{lem}
\label{auxlem2} Let $G,H$ be graphs and \(\varphi: V(G) \to V(H)\) be a surjective homomorphism. Then for any weight function \(\omega\), if \((H,\omega_{\varphi})\) has a sparse half, then so does  \((G,\omega)\).
\end{lem}

\begin{proof} Let $\pl{s}_H$ be a sparse half on \((H,{\omega}_{\varphi})\). Define
$$\pl{s}_G(v) := \frac{\omega(v)}{\omega_{\varphi}(\varphi(v))}\pl{s}_H(\varphi(v))$$
for every $v \in V(G)$. It is easy to check that \(\pl{s}_G\) is a sparse half of \(G\).
\end{proof}

Now let us introduce a family of graphs that plays a key role in our results. For an integer $d \geq 1$, let   \emph{\(F_d\)} be  the graph with \[V(F_d) = \{v_1, v_2, . . . , v_{3d -1}\},\] such that the vertex \(v_j\) has neighbors \(v_{j + d}, \dots , v_{j + 2d - 1}\) (all indices larger than \(3d-1\) are taken modulo \(3d - 1\)). It is easy to see that \(F_1=K_2\), \(F_2=C_5\) and \(F_3\) is a M\"{o}bius ladder (see Figure~\ref{mobius}). Clearly, if a graph is of \(F_d\)-type, it is of \(F_k\)-type, for all \(k>d\). These graphs were first introduced by Woodall~\cite{woodall}. In 1973 he conjectured that if the binding number of a graph $G$ is at least $3/2$ then
$G$ contains a triangle. Graphs $F_d$ are in the family of
graphs constructed by Woodall to explain that $3/2$ in his conjecture is the
least possible. Let us state some general properties of these graphs that we will use implicitly or explicitly throughout the paper. These can be easily checked. 

\begin{fact}\label{fdgraphs}For every \(d\geq 1\), $F_d$ is triangle-free, \(3\)-colorable and the only maximum independent sets are vertex neighborhoods, in particular, \(\alpha(F_d)=d\).
\end{fact}
In 1974 Andrasfai, Erd\H{o}s and Sos~\cite{erdossos} showed that every triangle-free graph $G$ of
order $n$ with $\delta(G) > 2n/5$ is bipartite (in other words, \(F_1\)-type). In a similar spirit, H\"{a}ggkvist~\cite{haggkvist} proved that every triangle-free graph $G$ of order $n$ with $\delta(G) > 3n/8$
is of $F_2$-type (i.e. $C_5$-type). In 1995  Jin~\cite{jin} improved H\"{a}ggkvist's result as following.

 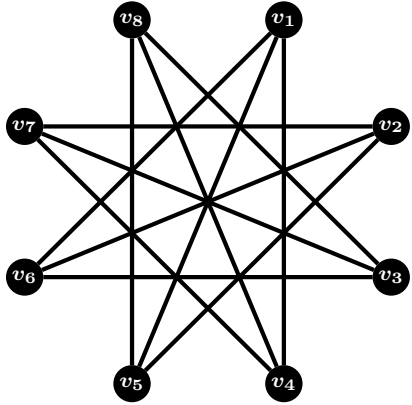
\begin{figure}
\label{mobius}
\begin{center}

\begin{tikzpicture}[label distance=3mm, line width=2]

every node/.style=draw,
    every label/.style=draw

    \tikzstyle{every node}=[draw,circle,fill=black,minimum size=12pt,
                            inner sep=0pt]

    \draw (0,0) node (v1) [draw=black] {\textcolor{white}{\scriptsize$\boldsymbol{v_1}$}}
           ++(315:2cm) node (v2) [draw=black] {\textcolor{white}{\scriptsize$\boldsymbol{v_2}$}}
           ++(270:2cm) node (v3) [draw=black] {\textcolor{white}{\scriptsize$\boldsymbol{v_3}$}}
           ++(225:2cm) node (v4) [draw=black] {\textcolor{white}{\scriptsize$\boldsymbol{v_4}$}}
           ++(180:2cm) node (v5) [draw=black] {\textcolor{white}{\scriptsize$\boldsymbol{v_5}$}}
           ++(135:2cm) node (v6) [draw=black] {\textcolor{white}{\scriptsize$\boldsymbol{v_6}$}}
           ++(90:2cm) node (v7) [draw=black]{\textcolor{white}{\scriptsize$\boldsymbol{v_7}$}}
           ++(45:2cm) node (v8) [draw=black]{\textcolor{white}{\scriptsize$\boldsymbol{v_8}$}};

     \draw [ultra thick] (v1)--(v4);
     \draw [ultra thick] (v1)--(v5);
     \draw [ultra thick] (v1)--(v6);
     \draw [ultra thick] (v2)--(v5);
     \draw [ultra thick] (v2)--(v6);
     \draw [ultra thick] (v2)--(v7);
     \draw [ultra thick] (v3)--(v6);
     \draw [ultra thick] (v3)--(v7);
     \draw [ultra thick] (v3)--(v8);
     \draw [ultra thick] (v4)--(v7);
     \draw [ultra thick] (v4)--(v8);
     \draw [ultra thick] (v5)--(v8);

\end{tikzpicture}
\end{center}
\caption{The M\"{o}bius ladder \(F_3\)}
\end{figure}
\begin{thm}[Jin,~\cite{jin}]\label{jin}Every triangle-free graph \(G\) of order \(n\) with minimum degree \(\delta(G) > 10n/29\) is of \(F_9\)-type. 
\end{thm}
Note that the results of Andrasfai, Erd\H{o}s and Sos~\cite{erdossos},
H\"{a}ggkvist~\cite{haggkvist} and Jin~\cite{jin} are sharp. In~\cite{jinchenkoh} the authors extended all the results mentioned above and proved the following stronger statement.
\begin{thm}[Chen, Jin, Koh, \cite{jinchenkoh}]\label{chenjinkohoriginal} Let \(d\geq 1\), \(G\) be a triangle-free graph of order \(n\) with \(\chi(G)\leq 3\) and \(\delta(G)>\left\lfloor\frac{(d+1)n}{3d+2} \right\rfloor\), then \(G\) is of \(F_d\)-type.
\end{thm}
For our purposes, we use the following statement which follows from Theorem~\ref{jin}, Theorem~\ref{chenjinkohoriginal} and Fact~\ref{fdgraphs}.
\begin{cor}[Chen, Jin, Koh ~\citep{jinchenkoh},~\cite{jin}] 
\label{chenjinkoh}
Every triangle-free graph \(G\) of order \(n\) with \(\delta(G)\geq\frac{5}{14}n\) is of $F_5$-type.
\end{cor}

To use Lemma~\ref{auxlem1} and Corollary~\ref{chenjinkoh} together we need to make sure that the homomorphism in the statement of Corollary~\ref{chenjinkoh} is surjective. Fortunately, this is not difficult if we allow ourselves to change the target graph.

\begin{lem}
\label{homomorphism}
Let \(\varphi\) be a homomorphism from graph \(G\) to \(F_d\), \(d\geq 2\). Then either \(\varphi\) is  surjective or there exists a homomorphism \({\varphi}'\) from \(G\) to \(F_{d-1}\).
\end{lem}
\begin{proof} Suppose \(\varphi\) is a homomorphism from graph \(G\) to \(F_d\) that is not surjective. Let \(V_i = {\varphi}^{-1}(v_i)\) for all \(i=1,2,\dots, 3d-1\). Without loss of generality suppose \(V_1\) is empty. Define  a mapping \(\varphi': V(G) \to V(F_{d-1})\)  as follows,
\[{\varphi}'(v) = \begin{cases} v_{i-1}, &\mbox{ if }  v\in V_i \textit{ and } 2\leq i \leq d-1,\\
v_{d-1}, & \mbox{ if }  v\in V_d \cup V_{d+1}, \\
v_{i-2},&\mbox{ if } v\in V_i \textit{ and }  d+2\leq i \leq 2d-2,\\
v_{2d-3},&\mbox{ if } v\in V_{2d-1} \cup V_{2d},\\
v_{i-3,}&\mbox{ if } v\in V_i \textit{ and }  2d+1\leq i \leq 3d-1.
 \end{cases}\]
It is easy to check that $\varphi'$ is a homomorphism from \(G\) to \(F_{d-1}\).
\end{proof}

In the next section we show  that for $1 \leq d \leq 4$ the weighted graph \((F_d,\omega)\) with minimum degree at least \(5/14\), has a sparse half for any weight function \(\omega\). In particular, if $\varphi: V(G) \to F_d$ is a surjective homomorphism, then the weighted graph \((F_d,\omega_{\varphi})\) has a sparse half. By Lemma~\ref{auxlem1} this implies that graph \(G\) has a sparse half. Hence Theorem~\ref{maintheorem1} will follow from the results of the next section.

\section{Weighted triangle-free graphs with minimum degree \(\geq \frac{5}{14}\) }\label{sec:minimum}

\begin{thm}
\label{weightedthm} For all  $1\leq d \leq 5$, if \((F_d,\omega)\) is a weighted graph with minimum degree at least \(5/14\), then it has a sparse half.
\end{thm}
\begin{proof}
The argument is separated into cases based on the value of \(d\).

\vskip 5pt \noindent
\textbf{ d=1:}  Suppose \(V(F_1)=\{v_1,v_2\}\) then since \(\omega(v_1) +\omega(v_2) = 1\), either \(\omega(v_1)\geq \frac{1}{2}\) or \(\omega(v_2) \geq \frac{1}{2}\), therefore \(v_1\) or \(v_2\) supports a sparse half in \((F_1,\omega)\).

\vskip 5pt \noindent
\textbf{ d=2:} Let \(V(F_2) = \{v_1, v_2, \dots, v_5\}\). If any two consecutive vertices together have total weight at least \(1/2\), then they induce an independent set, which means that they support a sparse half.

Now suppose that no two consecutive vertices have total weight at least \(1/2\), then any three consecutive vertices have total weight at least \(1/2\).
We define the following halves \(\pl{s}_i\) for each \(i=1,2,\dots, 5\) on the vertices of the graph and prove that there is at least one sparse half among them.
\[\pl{s}_i(v) = \begin{cases} \omega(v), &\mbox{ if }  v=v_i \textit{ or } v=v_{i+1},\\
\frac{1}{2} - \left(\omega(v_{i}) +  \omega(v_{i+1})\right) , & \mbox{ if }  v=v_{i+2}, \\
0,&\mbox{ otherwise.}
 \end{cases}\]

\begin{figure}[H]
\label{K3,3}
\begin{center}
\begin{tikzpicture}[label distance=3mm, line width=2]

every node/.style=draw,
    every label/.style=draw

\tikzstyle{every node}=[draw,circle,fill=black,minimum size=12pt,
                            inner sep=0pt]

  \node[shape=circle split,
    draw=gray!10,
    line width=0mm,text=white,font=\bfseries,
    circle split part fill={white,black}
    ] at (0,0) {\tiny\strut};

    \draw (0,0) node (v3) [draw=black,fill=none,label={306:$v_3$}]     {}
           ++(72:2cm) node (v2) [draw=black,fill=white, label={18:$v_2$}] {}
           ++(144:2cm) node (v1) [draw=black,fill=white, label={90:$v_1$}] {}
           ++(216:2cm) node (v5) [draw=black, label={162:$v_5$}] {0}
           ++(288:2cm) node (v4) [draw=black, label={234:$v_4$}] {0};

     \draw (v1)--(v3);
     \draw (v1)--(v4);
     \draw (v2)--(v4);
     \draw (v2)--(v5);
     \draw (v3)--(v5);

\end{tikzpicture}
\end{center}
\caption{\small A sparse half in uniformly weighted \(C_5\)}
\end{figure}
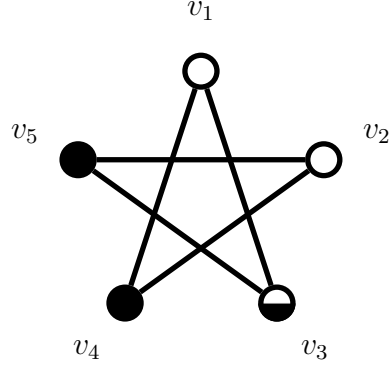

Note that
\begin{equation}
\label{c5sparse}
\pl{s}_i(E(G)) = \omega(v_i) \left(1/2 - \left(\omega(v_i) + \omega(v_{i+1})\right)\right).
\end{equation}
Summing up the equations~(\ref{c5sparse}) over all \(i = 1, \dots, 5\), we get
\begin{align*}
\label{5thcycle}
\sum_{i=1}^5{\pl{s}_i(E(G))} &= 1/2 - \sum_{i=1}^{5}{\omega(v_i)\left(\omega(v_i) + \omega(v_{i+1})\right)} \\
&=\frac{1}{2}-\frac{1}{2}\sum_{i=1}^{5}{\left(\omega(v_i) + \omega(v_{i+1})\right)^2} \\
&\leq \frac{1}{2}-\frac{1}{2}\cdot 5 \cdot\frac{4}{25}= \frac{1}{10},
\end{align*}
using Jensen's inequality for the function $x^2$. Thus one of the functions \(\pl{s}_i\) is a sparse half. Note that the proof above did not use the minimum degree condition, while we use it in the other two cases.

\vskip 5pt \noindent
\textbf{d=3:} Let \(F_3 = \{v_1, v_2, \dots, v_8\}\). As the minimum degree of \((F_3,{\omega})\) is at least \(5/14\), we have
\begin{equation}
\label{eqf3}
\omega(v_j )+ \omega(v_{j+1}) +\omega(v_{j+2}) \geq\delta(F_3,{\omega})  \geq 5/14,
\end{equation}
for all \(j = 1,2,\dots,8\). Summing these inequalities for \(j=i,i+3\) and \(j=i+6\), we  obtain \(\omega(v_j) \geq 1/14\) for \(j = 1,2,\dots,8\).

As in the case of $d=2$, if there exist three consecutive vertices of total weight at least \(1/2\), then we are done, since they induce an independent set. Suppose not, then every five consecutive vertices  have total weight more than \(1/2\). We define the following halves \(\pl{s}_i\), for each \(i=1,2,\dots, 8\) on the vertices of the graph and prove that there is at least one sparse half among them.

\[\pl{s}_i(v) = \begin{cases}
\omega(v) , & \mbox{ if }  v=v_{i+1}, v_{i+2},v_{i+3}   \\
\frac{1}{2}\left( {\frac{1}{2} - \left(\omega(v_{i+1}) +  \omega(v_{i+2}) +\omega(v_{i+3})\right) } \right), &\mbox{ if }  v=v_i \textit{ or } v=v_{i+4},\\
0,&\mbox{ otherwise.}
 \end{cases}\]
\begin{claim} For each \(i =1,2,\dots,8\), \(\pl{s}_i\) is a  half.
\end{claim}
\begin{proof}
It suffices to show that
\[\frac{1}{2}\left( {\frac{1}{2} - (\omega(v_{i+1}) +  \omega(v_{i+2}) +  \omega(v_{i+3})) } \right) \leq \omega(v_i) ,\]
\[\frac{1}{2}\left( {\frac{1}{2} - (\omega(v_{i+1}) +  \omega(v_{i+2}) +  \omega(v_{i+3})) } \right) \leq \omega(v_{i+4}). \]
By symmetry it suffices to prove the first inequality.   By~(\ref{eqf3}) we get that
\[\frac{1}{2}\left( {\frac{1}{2} - (\omega(v_{i+1}) +  \omega(v_{i+2}) + \omega(v_{i+3})) } \right) \leq \frac{1}{2}\cdot \left(\frac{1}{2}- \frac{5}{14}\right) = \frac{1}{14} \leq \omega(v_i).\]
This finishes the proof of the claim.
\end{proof}

We relegate the proof of the following lemma, which ensures that one of the halves $\pl{s}_i$ is sparse, to the Appendix.
\begin{lem}
\label{8cycletechnicallem} Let \(1/14\leq x_i\leq 1\) for \(i=1,2,\dots,8\). If \(\sum_{i=1}^{8}{x_i} = 1\) and \(x_i + x_{i+1} + x_{i+2} \geq 5/14\) for every \(i\), then there exists \(i\) such that
\begin{equation}
\label{8cycle}
\frac{1}{2}\left(\frac{1}{2}-(x_i+x_{i+1}+x_{i+2})\right)\left(x_i +x_{i+2}\right) + \frac{1}{4}\left(\frac{1}{2}-(x_i+x_{i+1}+x_{i+2})\right)^2 \leq \frac{1}{50}.
\end{equation}
\end{lem}

\noindent \textbf{d=4:}  Let \(F_4 = \{v_1, v_2, \dots, v_{11}\}\). The minimum degree condition gives us the following inequality
\begin{equation}
\label{eq1}
\omega(v_i) + \omega(v_{i+1}) +\omega(v_{i+2}) + \omega(v_{i+3}) \geq 5/14,
\end{equation}
for all \(i = 1,2,\dots,11\). It follows, as in the case $d=3$, \(\omega(v_i)\geq 1/14\) for all \(i\).  We may assume that no four consecutive vertices have total weight at least \(1/2\), otherwise, we are done - these four vertices induce an independent set. Therefore, we can assume that 
 \[\omega(v_i) +\omega(v_{i+1}) + \dots + \omega(v_{i+6}) > 1/2,\]
for all \(i=1,2,\dots, 11\). This allows us to define halves $\pl{s}_i$  in the following way:
\[\pl{s}_i(v) = \begin{cases}
\omega(v) , & \mbox{ if }  v=v_{i+1}, v_{i+2},v_{i+3}, v_{i+4},  \\
\frac{1}{2}\left( {\frac{1}{2} - \left(\omega(v_{i+1}) +  \omega(v_{i+2}) +\omega(v_{i+3}) +\omega(v_{i+4}) \right) } \right), &\mbox{ if }  v=v_i \textit{ or } v=v_{i+5},\\
0,&\mbox{ otherwise.}
 \end{cases}\]

As in the previous case, it is easy to verify that each $\pl{s}_i$ is a half, and the following lemma proved in the Appendix implies that at least one of these halves is sparse.

\begin{lem}
\label{11cycletechnicallem}
Suppose \(x_1, x_2, \dots , x_{11}\) are reals such that \(1/14\leq x_i\leq 1\) for each \(i=1,2,\dots,11\) and \(\sum_{i=1}^{11}{x_i} = 1\). If \(x_i + x_{i+1} + x_{i+2}  + x_{i+3}\geq 5/14\) for every \(i\), then there exists \(i\) such that
 \begin{align*}
&\frac{1}{2}\left(\frac{1}{2} - (x_{i+1} +x_{i+2}+x_{i+3} + x_{i+4})\right)(x_{i+1} + x_{i+4}) \\
&+ \frac{1}{4}\left(\frac{1}{2} - (x_{i+1} +x_{i+2}+x_{i+3} + x_{i+4})\right)^2 \leq 1/50.
\end{align*}
\end{lem}
\noindent \textbf{d=5:}  Let \(F_5 = \{v_1, v_2, \dots, v_{14}\}\). As in the previous cases, the minimum degree condition gives us \(\omega(v_i) \geq \frac{1}{14}\) for every \(i=1, 2,\dots, 14\). But on the other hand, \(\sum_{i=1}^{14}{\omega(v_i)} = 1\), hence \(\omega(v_i) = 1/14\), for each \(i=1, 2, \dots, 14\). Define \[\pl{s}(v) = \begin{cases}
\omega(v) , & \mbox{ if }  v=v_{1}, v_2,v_3,v_4,v_5,v_6,v_7  \\
0,&\mbox{ otherwise.}
 \end{cases}\]
It is easy to check that \(\pl{s}\) is a half. In fact, it is a sparse half. Indeed, \(\pl{s}(E(G)) = \frac{3}{14^2}\leq \frac{1}{50}\).\end{proof}
\begin{proof}[Proof of Theorem~\ref{maintheorem1}]
Let $G$ be a triangle-free graph on $n$ vertices with minimum degree $\geq 5n/14$. By Lemma~\ref{auxlem1} it suffices to prove that the uniformly weighted graph $(G,\omega)$ has a sparse half. By Theorem~\ref{chenjinkoh} and Lemma~\ref{homomorphism}, the graph $G$ admits a surjective homomorphism $\varphi$ to $F_d$ for $1 \leq d\leq 5$.
Clearly, $(F_d,\omega_{\varphi})$ has minimum degree $\geq 5/14$ and thus has a sparse half by Theorem~\ref{weightedthm}. Theorem~\ref{maintheorem1} now follows from Lemma~\ref{auxlem2}.
\end{proof}

\section{Uniform sparse halves, balanced weights and disturbed subgraphs}\label{s:uniform}

The purpose of this section is to develop a set of tools, that,  under fairly general circumstances, allow us to show that graphs ``close" to a fixed graph have sparse halves. A number of technical definitions will be necessary. We start by a variant of the definition of \emph{edit distance} (see e.g.~\cite{lovasz}).

\begin{definition}\label{d:epsClose}
Given graphs \(G\) and \(H\) of orders \(n\) and \(k\) respectively, we say that \(G\) can be \(\eps\)-\emph{approximated} by \(H\), if there exists a partition  \(\mathcal{V} = \{V_1, V_2, \dots, V_k\}\)  of \(V(G)\) such that
\[\left||V_i|-\frac{n}{k}\right| \leq \eps n, \textrm{ for all } i=1,2,\dots k, \textrm { and } \pl{e}(G \triangle H_{\mathcal{V}}) \leq \eps n^2,\]where
\(H_{\mathcal{V}}\) is a graph on vertex set \(V(G)\) with adjacency defined as follows. Each \(H_{\mathcal{V}}[V_i]\) is an independent set, \((V_i, V_j)\) induces a complete bipartite graph in \(H_{\mathcal{V}}\) if \(v_iv_j \in E(H)\) and an empty graph, otherwise.
\end{definition}

We will also need a stronger notion of distance defined as follows.

\begin{definition}
Given a graph \(G\) and \(F\subseteq E(G)\), we say that \(D\subseteq V(G)\) is an \emph{\(\eps\)-covering} set for \(F\), if \(|D|\leq \eps \pl{v}(G)\) and every edge in \(F\) has at least one end in \(D\). 
\end{definition}
\begin{definition}
We say that the graph \(G\)  is an \emph{\(\eps\)-disturbed} subgraph of \(G'\) for some \(0<\eps<1\), if the following conditions hold:
\begin{enumerate}
\item \(V(G)= V(G')\),
\item there exists an \(\eps\)-covering set for \(E(G)\setminus E(G')\) in $G$,
\item \(|N_{G}(v) \setminus N_{G'}(v)|\leq \eps \pl{v}(G')\) for every vertex \(v\in V(G').\)
\end{enumerate}
\end{definition}

Note that if \(G\) is an \(\eps\)-disturbed subgraph of some graph \(G'\), it does not imply that \(G\) is a subgraph of \(G'\), as one might expect. However, from our definition, it follows that \(G\) is not "far" from some subgraph of \(G'\) and that is why we use \(\eps\)-disturbed term.

For a graph \(H\), let $\mc{N}(H)$ be the set of neighborhoods of vertices of $H$ and $\mc{I}(H)$ be the set of maximum independent sets of $H$. Let $\mc{I}^{*}(H) = \mc{I}(H) \setminus \mc{N}(H)$. We construct a graph $H^{*}$ as follows. Let $V(H^{*})=V(H) \cup \mc{I}^{*}(H)$, $H$ be an induced subgraph of $H^{*}$, and let every $I \in \mc{I}^{*}(H)$ be adjacent to every $v \in I$ and no other vertex of $H^{*}$.
We say that a weighted graph $(H^{*},\omega)$ is \emph{$\eps$-balanced} if $|\omega(v) - 1/\pl{v}(H)| \leq \eps$ for every $v \in V(H)$ and $\omega(v) \leq \eps$ for every $v \in V(H^{*})\setminus V(H)$. We are  now ready for our first technical result.

\begin{thm}\label{t:DisturbedGeneral} Let $H$ be a maximal triangle-free graph. For every $\eps>0$ there exists $\delta>0$ such that if $G$ is a triangle-free graph which can be $\delta$-approximated by $H$
then there exists a graph $G'$ and a homomorphism $\varphi$ from $G'$ to $H^{*}$ with the following properties.
\begin{enumerate}
\item[(i)] $G$ is an $\eps$-disturbed subgraph of $G'$,
\item[(ii)] $(H^{*},\omega_\varphi)$ is $\eps$-balanced, where \(\omega\) is the uniform weight function defined on \(G\),
\item[(iii)] $\varphi$ is a \emph{strong homomorphism}, that is, $uv \not \in E(G')$ implies $\varphi(u)\varphi(v) \not \in E(H^{*})$ for every pair of vertices $u,v \in V(G')$.
\end{enumerate}
\end{thm}
\begin{proof}
We assume that $V(H)=\{1,2,\ldots,k\}$. We show that $\delta>0$ satisfies the theorem if $(k+2)\sqrt{\delta} \leq \min(\eps,1/k)$.
Let $\mc{V}=(V_1,V_2,\ldots,V_k)$ be the partition of $V(G)$, and the graph $H_{\mc{V}}$ be as in Definition~\ref{d:epsClose}. Let $n:=\pl{v}(G)$,  $F:=E(G) \triangle E(\mc{H}_\mc{V})$ and $J$ be the set of all vertices of $G$ incident to at least $\sqrt{\delta}n$ edges in $F$. We have $\frac{1}{2}\sqrt{\delta}n|J| \leq |F| \leq \delta n^2 $. It follows that $|J| \leq 2\sqrt{\delta}n$.

We define a map $\varphi: V(G) \to V(H^{*})$, as follows. If $v \in V_i\setminus J$ for some $i \in V(H)$ then $\varphi(v):=i$. Now consider $v \in J$ and let $$I_0(v) :=\{i \in V(G) \: | \: |N(v) \cap V_i| > \sqrt{\delta}n\}.$$
Then $I_0(v)$ is independent, as otherwise there exist $i,j \in [k]$, such that $ij \in E(H)$, $|N(v) \cap V_i| > \sqrt{\delta}n$ and  $|N(v) \cap V_j| > \sqrt{\delta}n$. As $G$ is triangle-free it follows that $\pl{e}(H_{\mc{V}}\triangle G) > \eps n^2$, contradicting the choice of $\mc{V}$. Let $I(v)$ be a maximal independent set containing $I_0(v)$, chosen arbitrarily. Define $\varphi(v):=i$, if $I(v)=N_H(i)$ for some $i \in V(H)$, and $\varphi(v):=I(v)$, otherwise.

Let $G'$ be the graph with $V(G')=V(G)$ and the vertices $uv \in E(G')$ if and only if $\varphi(u)\varphi(v) \in E(H^{*})$. Then $\varphi$ is a strong homomorphism from $G'$ to $H^{*}$. For $i \in V(H)$ we have $V_i \setminus J \subseteq \varphi^{-1}(i) \subseteq V_i \cup J$ and, therefore, $||\varphi^{-1}(i)|/n - 1/k| \leq \delta +2\sqrt{\delta}.$ For $I \in V(H^{*})\setminus V(H)$ we have $|\varphi^{-1}(I)| \leq |J| \leq 2\sqrt{\delta}n$. Thus (ii) holds, as  $\eps \geq \delta +2\sqrt{\delta}.$

It remains to verify (i). We will show that $J$ is an $\eps$-covering set for $F':=E(G)\setminus E(G')$ for $\delta$ sufficiently small. First, we show that every edge of $F'$ has an end in $J$. Indeed suppose that $uv \in E(G)$ for some $u \in V_i,v \in V_j$ with $i,j \in V(H)$ not necessarily distinct and $ij \not \in E(H)$. Then there exists $h \in V(H)$ adjacent to both $i$ and $j$, as $H$ is maximally triangle-free. It follows that both $v$ and $u$ have at least $(1/k -\delta -\sqrt{\delta})n$ neighbors in $V_h$ and share a common neighbor if $2(1/k - \delta  - \sqrt{\delta}) > 1/k$. Thus our first claim holds, as $\delta + \sqrt{\delta} < 1/k$.

Consider now $v \in J$ and let $N'(v)$ be the set of neighbors of $v$ in $V(G)\setminus V(J)$ joined to $v$ by edges of $F'$. Then $|N'(v) \cap V_i| \leq \sqrt{\delta}n$ for every $i \in V(H)$ by the choice of $\varphi(v)$. It follows that $|N'(v)| \leq k\sqrt{\delta}n$. Therefore $v$ is incident to at most $(k+2)\sqrt{\delta}n$ edges in $F$, as $|J| \leq 2 \sqrt{\delta}n.$ Thus $J$ is  an $\eps$-covering set for $F'$, as $(k+2)\sqrt{\delta}n \leq \eps.$
\end{proof}

We say that $H$ is \emph{entwined} if $\mc{I}^{*}(H)$
is intersecting. Note that, using Fact~\ref{fdgraphs}, one can see that $F_d$ is entwined for every $d\geq1$,  as $\mc{I}^{*}(F_d)$ is empty. It is routine to check that the Petersen graph is entwined.
We say that a graph \(G\) of order \(n\)  is \emph{\(c\)-maximal triangle-free}  if it is triangle-free and adding any new edge to \(G\) creates at least \(cn\) triangles. For a weighted graph \((G,\omega)\), we say that it is \emph{\(c\)-maximal triangle-free}  if \(G\) is triangle-free and adding any new edge to \(G\) creates triangles of total weight at least \(c\).

\begin{lem}\label{l:EntwinedToCMaximal}
Let $H$ be an entwined maximal triangle-free graph, and let $0<\eps<1/\pl{v}(H)$. If $(H^{*},\omega)$ is $\eps$-balanced then it is $(1/\pl{v}(H)-\eps)$-maximal triangle-free.
\end{lem}

\begin{definition}
For a weighted graph \((G,\omega)\) we call a distribution \(\textbf{s}\) defined on  the set of halves of $(G,\omega)$
a \emph{\(c\)-uniform sparse half} for some \(0< c\leq 1\), if
\begin{enumerate}
\item for every edge \(e\in E(G)\), \(\E{\textbf{s}(e)} \geq c\:{\omega(e)}\),
\item \(\E{\textbf{s}(E(G))} \leq \frac{1}{50}\).
\end{enumerate}
\end{definition}
Whenever we refer to $c$-uniform sparse halves in unweighted graphs, they are understood as the $c$-uniform sparse halves in the corresponding uniformly weighted graphs.

\begin{thm}
\label{cuniformtouniform}
Let \(0<c<1\) be real, \(G'\) be a \(c\)-maximal triangle-free graph and \(G\) be a triangle-free \(\frac{c^2}{2(1+c)}\)-disturbed subgraph of \(G'\). If $G'$ has a $c$-uniform sparse half then \(G\) has a sparse half.
\end{thm}

\begin{proof} Let \(\eps = \frac{c^2 }{2(1+c)}\), \(F:=E(G)-E(G')\) and let \(M\) be a maximal matching in \(F\). Suppose \(|M| = \delta n\) for some \(0\leq \delta \leq 1/2\). Since \(G\) is an \(\eps\)-disturbed subgraph of \(G'\), it has an \(\eps\)-covering set for \(F\). Let \(D\) be the minimum one. By the choice of $M$ every edge of $F$ has at least one end in $V(M)$. It follows that $|D| \leq |V(M)|=2|M|$. By the third condition in the definition of \(\eps\)-disturbed subgraph,
$|F| \leq \eps n |D|\leq 2\delta \eps n^2.$

Let \(F' := E(G')\setminus E(G)\).  For an edge $e \in M$, let $T(e)$ be the set of edges $e' \in F'$, such that the only vertex of $V(M)$ that $e'$ is incident to, is an end of $e.$ Let $u,v$ be the ends of $e$. By the \(c\)-maximality of \(G'\), we have $|N_{G'}(u) \cap N_{G'}(v)| \geq cn$ for every pair of vertices $u,v \in V(G')$ non-adjacent in $G'$. Since \(e \in E(G)\) and $G$ is triangle-free, for every vertex $w \in  N_{G'}(u) \cap N_{G'}(v)$, either \(uw \in F'\) or \(vw \in F'\). It follows that $|T_e| \geq |N_{G'}(u) \cap N_{G'}(v)| - |V(M)| \geq (c -2 \delta)n$. Thus $|F'| \geq  (c-2\delta)n \cdot \delta n$.

Let \(\textbf{s}\) be a \(c\)-uniform sparse half in the graph \(G'\). Then
\begin{align*}
\E{\textbf{s}(E(G)) - \textbf{s}(E(G'))} &= \E{\textbf{s}(F) - \textbf{s}(F')} \\
&= \E{\textbf{s}(F)} - \E{\textbf{s}(F')},
\end{align*}
by linearity of the expectation. We have
\[\E{\textbf{s}(F')}  = \sum_ {e\in F'}{\E{\textbf{s}(e)}}  \geq \sum_ {e\in F'}{c\:\omega(e)}  =  c \sum_ {e\in F'}{\frac{1}{n^2}}  \geq  \frac{c| F'| }{n^2}.
\]
On the other hand,
\[\E{\textbf{s}(F)}  = \sum_ {e\in F}{\E{\textbf{s}(e)}}  \leq \sum_ {e\in F}{ \omega(e)}  = \frac{|F|}{n^2}.
\]
Finally, note that \(\delta\leq \eps ,\) since every edge in \(M\) has at least one of its ends in \(D\).
Hence,
\begin{align*}
\E{\textbf{s}(E(G)) - \textbf{s}(E(G'))} &\leq  \frac{|F|}{n^2} -  \frac{c| F'|}{n^2} \leq 2\delta \eps  -  c (c-2\delta)\delta \\
&= \delta(2\delta c + 2\eps-c^2) \leq 0,
\end{align*}
where the last inequality holds, as $2\delta c + 2\eps-c^2 \leq 2(1+c)\eps -c^2=0$.
Therefore,
$\E{\textbf{s}(E(G))} \leq \E{\textbf{s}(E(G'))} \leq \frac{1}{50},$
and the graph \(G\) has a sparse half by Lemma~\ref{auxlem1}.
\end{proof}

We are now ready to prove the main result of this section.

\begin{thm}\label{t:CloseMain} Let $H$ be an entwined maximal triangle-free graph. Suppose that there exists $\alpha>0$ such that, if $(H^{*},\omega)$ is $\alpha$-balanced, then $(H^{*},\omega)$ has an $\alpha$-uniform sparse half. Then there exists $\delta>0$ such that every triangle-free graph $G$ which can be $\delta$-approximated by $H$ has a sparse half.
\end{thm}
\begin{proof} Let $\eps:=\min\left(\frac{1}{3\pl{v}(H)^2},\frac{\alpha^2}{2(1+\alpha)}\right)$, and let $\delta$ be chosen so that there exist $\varphi$ and $G'$ satisfying the conclusion of Theorem~\ref{t:DisturbedGeneral}. The weighted graph $(H^{*},\omega_{\varphi})$ has an $\alpha$-uniform sparse half, as $\eps \leq \alpha$. Therefore, the graph $G'$ has an $\alpha$-uniform sparse half.
By Lemma~\ref{l:EntwinedToCMaximal}, the graph $(H^{*},\omega_{\varphi})$ is $(1/\pl{v}(H)~-~\eps)$-triangle-free. Therefore, the graph $G'$ is $\eps$-triangle-free, because $\varphi$ is a strong homomorphism.  Let $c:=\min(\alpha,1/2\pl{v}(H))$. Then  $\eps \leq c^2/(2(1+c))$ and $c \leq 1/\pl{v}(H)-\eps$, by the choice of $\eps$. It follows that the conditions of Theorem~\ref{cuniformtouniform} are satisfied for $G'$ and $G$. Thus $G$ has a sparse half, as desired.
\end{proof}

\section{ Triangle-free graphs with at least \((1/5-\gamma)n^2\) edges}\label{sec:average}

In order to establish the conjecture for the triangle-free graphs with average degree \(\left(\frac{2}{5}-\gamma\right)n\) we separate the cases when the graphs under consideration are close in the sense of Definition~\ref{d:epsClose} to  the blowup of $C_5$ and when they are not. In the second case, we use the main result of Keevash and Sudakov from \cite{Sudakov}. Note that this statement follows from their proof method but it is not explicitly stated as such in~\cite{Sudakov}.

\begin{thm}
\label{theoremsudakov}
Let \(G\) be a triangle-free graph  on \(n\) vertices such that one of the following conditions holds
\begin{itemize}
\item[(a)] either \(\frac{1}{n}\sum_{v\in V(G)}{d^2(v)} \geq {\left(\frac{2}{5}n\right)}^2\) and \(\Delta(G) < \left (\frac{2}{5} + \frac{1}{135}\right)n\), or
\item[(b)] \(\Delta(G) \geq  \left(\frac{2}{5} + \frac{1}{135}\right)n\) and  \(\frac{1}{n}\sum_{v\in V(G)}{d(v)} \geq \left(\frac{2}{5} -\frac{1}{125} \right)n\).
\end{itemize}
Then $G$ has a sparse half.
\end{thm}

The next theorem represents the main technical step in the proof of Theorem~\ref{maintheorem2}.

\begin{thm}
\label{maintheorem}
For every $\eps > 0$ there exists $\delta>0$ such that the following holds.  If \(G\) is a triangle-free graph with $\pl{v}(G)=n$ and \(\pl{e}(G) \geq (\frac{1}{5}-\delta)n^2 \) then either
\begin{itemize}
\item [(i)] \(G\) can be \(\eps\)-approximated by $C_5$,
\item [(ii)] or at least \(\delta n\) vertices of \(G\) have degree at least \(\left(\frac{2}{5}+\delta\right) n\),
\item [(iii)] or there exists an \(F\subseteq E(G)\) with \(|F|\leq \eps n^2\) such that the graph $G - F$ is bipartite.
\end{itemize}
\end{thm}

The proof of Theorem~\ref{maintheorem} uses the following technical lemmas.

\begin{lem}
\label{mindegreelemma}  For \(\delta>0\) let \(G\) be a  graph with $\pl{v}(G)=n$ and $\pl{e}(G) \geq (\frac{1}{5}-\delta)n^2$. Then either
\begin{itemize}
\item[(1)] at least \(\delta n\) vertices have degree at least \((\frac{2}{5} + \delta)n\),
\item[(2)] or at most \(2\sqrt{\delta} n \) vertices have degree at most \((\frac{2}{5}-2\sqrt{\delta})n\).
\end{itemize}
\end{lem}

\begin{proof}
Suppose that the outcome (1) does not hold. Let \(tn\) be the number of vertices that have degree at most \((\frac{2}{5}-2\sqrt{\delta})n\). Then
\begin{align*}
2\left(\frac{1}{5}-\delta\right)n^2 &\leq 2\pl{e}(G) = \sum_{v\in V}{d(v)} \\
&=
\sum_{\substack{v\in V \\ d(v) \leq (\frac{2}{5}-2\sqrt{\delta})n} }{d(v)}
+\sum_{\substack{v\in V \\ (\frac{2}{5}-2\sqrt{\delta})n < d(v) < (\frac{2}{5}+\delta)n} }{d(v)} +
\sum_{\substack{v\in V \\ d(v) \geq (\frac{2}{5}+\delta)n} }{d(v)}
\\&< tn\left(\frac{2}{5}-2\sqrt{\delta}\right)n + \left(1-t\right)n\left(\frac{2}{5}+\delta\right)n + \delta n^2\\
&=\left(\frac{2}{5} - \sqrt{\delta}t +2\delta -t\delta\right)n^2.
\end{align*}
Thus $-2\delta < 2\delta -2t\sqrt{\delta} - \delta t,$
and
\[t < \frac{4\delta}{2\sqrt{\delta} +\delta} < 2\sqrt{\delta},\]
as desired.
\end{proof}

\begin{lem}
\label{upperlowerbounds}  Let $H$ be a graph on $n$ vertices with minimum degree at least \(\left(\frac{2}{5}-\delta\right)n\). Let $\varphi$ be a surjective homomorphism from $H$ to $C_5$. Then
\begin{equation}
\left(\frac{1}{5}-3\delta\right)n\leq |\varphi^{-1}(v)| \leq \left(\frac{1}{5} +2\delta\right)n,
\end{equation}
for every $v \in V(C_5)$.
\end{lem}

\begin{proof} Let the vertices of $C_5$ be labelled $v_1,v_2,\ldots, v_5$ such that the neighbors of \(v_i\) are \(v_{i+2}\) and \(v_{i+3}\). Define $V_i :=\varphi^{-1}(v_i)$ for $i=1,2,\ldots,5$. By the minimum degree condition on \(H\), we have
\begin{align}
\label{ineqvi}
|V_i| + |V_{i+1}| \geq  \left(\frac{2}{5}-\delta\right)n \\
|V_{i+2}| + |V_{i+3}| \geq  \left(\frac{2}{5}-\delta\right)n \\
\label{ineqviii}
|V_{i+4}| + |V_i| \geq  \left(\frac{2}{5}-\delta\right)n,
\end{align}
therefore
\[ |V_i| + \left(\frac{2}{5}-\delta\right)n + \left(\frac{2}{5}-\delta\right)n \leq |V_i| +  (|V_{i+1}| +  |V_{i+2}|) + (|V_{i+3}| +  |V_{i+4}|) \leq n, \]
which gives us the desired upper bound.

For the lower bound summing inequalities  (\ref{ineqvi})-(\ref{ineqviii})
we get
\[ n + |V_i| \geq \sum_{j=1}^5{|V_j|} + |V_i| \geq
3\left(\frac{2}{5}-\delta\right)n,\]
which gives us $|V_i|\geq \left(\frac{1}{5} -3\delta\right)n$, as desired.
\end{proof}

\begin{proof}[Proof of Theorem~\ref{maintheorem}] We show that $\delta := (\eps/40)^2$ satisfies the theorem for $\eps \leq 1$.  We apply Lemma~\ref{mindegreelemma}. The first outcome of Lemma~\ref{mindegreelemma} corresponds to the outcome (ii) of the theorem. Therefore we assume that the second outcome holds: There exists \(|S|\leq \frac{\eps}{20}n\) such that every vertex in $V(G)\setminus S$ has degree at least \(\left(\frac{2}{5}-\frac{\eps}{20}\right)n\) in $G$. It follows that $G':=G - S$ has the minimum degree at least \(\left(\frac{2}{5}-\frac{\eps}{10}\right)n\).

Since $\left(\frac{2}{5}-\frac{\eps}{10}\right)n \geq \frac{3}{8}|V(G')|$, Theorem~\ref{chenjinkohoriginal} implies that there exists a homomorphism $\varphi$ from $G'$ to $C_5$. Let \(V(C_5) = \{v_1, v_2, \dots, v_5\}\) and  \(V_i= {\varphi}^{-1}(v_i)\) for each \(i=1, 2,\dots,5\).

If the homomorphism \(\varphi\) is not surjective then by Lemma~\ref{homomorphism} the graph \(G'\) is bipartite and therefore the graph \(G^{*}= (V(G), E(G'))\) is also bipartite. We have $\pl{e}(G)-\pl{e}(G') \leq |S|n \leq \eps n^2.$ Thus outcome (iii) holds. Hence we can suppose that the homomorphism \(\varphi\) is surjective.  Applying Lemma~\ref{upperlowerbounds} to the graph \(G'\) with \(\delta = \frac{\eps}{10}\), we get that
\begin{equation}
\label{lowerupperbound}
\left(\frac{1}{5}-\frac{2\eps}{5}\right)n\leq |V_i| \leq \left(\frac{1}{5} +\frac{\eps}{5}\right)n.
\end{equation}
Let $\mc{V}:=(V_1\cup S,V_2,\ldots,V_5)$ be a partition of $V(G)$. From (\ref{lowerupperbound}) we have $\left||V_i|-\frac{n}{5}\right| \leq \eps n$.
Let $H_{\mc{V}}$ be as in Definition~\ref{d:epsClose}.  Then
\begin{align*}
\pl{e}(G \triangle H_{\mc{V}}) \leq |S|n + \pl{e}(H_{\mc{V}})-\pl{e}(G') \leq \frac{\eps}{20}n^2+\frac{n^2}{5}-\frac{1}{2}\left(\frac{2}{5}-\frac{\eps}{10}\right)&\left( 1-\frac{\eps}{20}\right)n^2\\
&\leq \eps n^2.
\end{align*}
Thus outcome (i) holds.
\end{proof}

If outcome (i) of Theorem~\ref{maintheorem} holds, our goal is to apply Theorem~\ref{t:CloseMain}. To do that we need to ensure that a $c$-balanced weighting of $C_5$  has a $c$-uniform sparse half for some $c>0$.

\begin{thm}
\label{uniformtheorem}
Any $(1/50)$-balanced weighted graph  $(C_5,\omega)$ has a $(1/30)$-uniform sparse half.
\end{thm}

\begin{proof} Let $\delta:=1/50$, \(V(C_5) = \{v_1,v_2,\dots, v_5\}\) and $E(C_5)=\{v_iv_{i+2}\}_{i=1}^5$, as in the proof of Theorem~\ref{weightedthm}. We define a distribution on the set of halves of the graph \(C_5\). Recall the  halves  \(\pl{s}_i\), \(1 \leq i \leq 5\) that we have defined earlier in the proof of the Theorem~\ref{weightedthm}.

Let the probability mass of the distribution \(\textbf{s}\) be \(\frac{1}{5}\) on every \(\pl{s}_i\), \(i =1,2,\dots, 5\). We show that \(\textbf{s}\) is a \(\frac{1}{30}\)-uniform sparse half. Let us begin by showing that $\E{\textbf{s}(e)}) \geq \frac{1}{30}\:{\omega(e)}$ for every $e \in E(C_5)$. Let \(e= (v_i,v_{i+2})\), then
\[\E{\textbf{s}(e)} = \frac{1}{5} \cdot \omega(v_i)\left(\frac{1}{2}-(\omega(v_i) + \omega(v_{i+1}))\right). \]
Hence,
\[\E{\textbf{s}(e)} \geq \frac{1}{5} \left(\frac{1}{5}-\delta\right)\left(\frac{1}{2}-2 \left(\frac{1}{5} +\delta\right)\right) = \frac{1}{5} \cdot  \left(\frac{1}{5}-\delta\right)\left(\frac{1}{10}-2\delta\right). \]
On the other hand,
\[\omega(e)= \omega(v_i)\cdot \omega(v_{i+2}) \leq \left(\frac{1}{5} +\delta\right)^2.\]
Thus it suffices to show that
$$\frac{1}{5} \left(\frac{1}{5}-\delta\right)\left(\frac{1}{10}-2\delta\right) \geq \frac{1}{30}\left(\frac{1}{5} +\delta\right)^2,$$
which can be easily verified. It is shown in the proof of Theorem~\ref{weightedthm} that $\E{\textbf{s}(E(G))} \leq \frac{1}{50}.$ Thus, \(\textbf{s}\) is a \(\frac{1}{30}\)-uniform sparse half of $G$, as claimed.
\end{proof}

We need a final technical lemma.

\begin{lem}\label{avgdegree} For every \(\delta >0\) there exists \(\gamma>0\) such that if \(G\) is a graph on $n$ vertices with at least  \(\left(\frac{1}{5}-\gamma\right)n^2\) edges and at least \(\delta n\) vertices of degree at least \((\frac{2}{5} + \delta)n\) then
 \[\frac{1}{n}\sum_{v\in V(G)}{d^2(v)} \geq {\left(\frac{2}{5}n\right)}^2.\]
\end{lem}
\begin{proof}
Suppose that $G$ contains \(\alpha n\) vertices of degree at most \(\frac{2}{5}n\) and \(\beta n\) vertices of degree at least \(\left(\frac{2}{5} +\delta\right)n\). We may assume that the average degree of $G$ is less than \(\frac{2}{5}n\), as otherwise lemma clearly holds.
Thus,
\[\frac{2}{5} n  >  \left(1-\alpha - \beta\right) \frac{2}{5}n+ \beta\left(\frac{2}{5}+\delta\right)n,\]
hence \(\alpha > \frac 52 \beta\delta \geq \frac 52 {\delta}^2.\) We have
\begin{align*}
n\sum_{v\in V(G)}{d^2(v)} - {\left(\sum_{v\in V(G)}{d(v)}\right)}^2 &= \frac{1}{2}\sum_{u\neq v}{{(d(u)-d(v))}^2} \\
&\geq \alpha \delta {\left(\frac{2}{5}+\delta - \frac{2}{5}\right)}^2n^4 > \frac 52{\delta}^5n^4.
\end{align*}
Hence
 \begin{align*}
 \frac{1}{n}\sum_{v\in V(G)}{d^2(v)} &\geq \frac{1}{n^2}{\left(\sum_{v\in V(G)}{d(v)}\right)}^2 + \frac 52 {\delta}^5n^2 \geq 4{\left(\frac{1}{5}-\gamma\right)}^2n^2 + \frac 52 {\delta}^5n^2\\
 &\geq \frac{4}{25}n^2 +\left( \frac 52{\delta}^5 -\frac{8}{5}\gamma \right)n^2 \geq \frac{4}{25}n^2,
 \end{align*}
 if we choose \(\gamma = \frac{25}{16}{\delta}^5\).
\end{proof}

\begin{proof}[Proof of Theorem~\ref{maintheorem2}] By Theorem~\ref{uniformtheorem}, $\alpha=1/50$ satisfies the conditions in the statement of Theorem~\ref{t:CloseMain} for $H:=C_5$. Thus by Theorem~\ref{t:CloseMain} there exists $0<\eps\leq 1/50$ such that every triangle-free graph $G$ that can be $\eps$-approximated by $C_5$ has a sparse half. Let $\delta$ be such that Theorem~\ref{maintheorem} holds, and finally let $\gamma$ be such that Lemma~\ref{avgdegree} holds. We show that Theorem~\ref{maintheorem2} holds for this choice of $\gamma$.

We distinguish cases based on the outcome of Theorem~\ref{maintheorem} applied to $G$.

\noindent\textbf{ Case (i):} If $G$ can be $\eps$-approximated by $C_5$ then the theorem holds by the choice of $\eps$.

\noindent\textbf{ Case (ii):} Now suppose that at least \(\delta n\) vertices of \(G\) have degree at least \(\left(\frac{2}{5}+\delta\right)n\). If \(\Delta(G) \geq \left(\frac{2}{5} + \frac{1}{135}\right)n \) then Theorem~\ref{theoremsudakov} (b) implies that there is a sparse half in \(G\). Therefore we  assume that \(\Delta(G) < \left(\frac{2}{5} + \frac{1}{135}\right)n \).
By Lemma~\ref{avgdegree} and the choice of $\gamma$ we have $\frac{1}{n}\sum_{v\in V(G)}{d^2(v)} \geq {\left(\frac{2}{5}n\right)}^2$. Hence Theorem~\ref{theoremsudakov} (a) implies that there is a sparse half in \(G\).

\noindent\textbf{ Case (iii):}  Lastly, suppose there exists an \(F\subseteq E(G)\) with \(|F| \leq \eps n^2\) such that the graph \(G'=(V(G), E(G)\setminus F)\) is bipartite with bipartition \((U,V)\).  Then either \(|U|\geq \frac{n}{2}\) or  \(|V|\geq \frac{n}{2}\). Without loss of generality suppose  \(|U|\geq \frac{n}{2}\) . The set \(U\) is independent in \(G'\), while in \(G\) it might not be, but
$\pl{e}(G[U]) \leq |F| \leq \eps n^2 \leq \frac{ n^2}{50}$. Hence \(U\) supports a sparse half in graph \(G\).
\end{proof}

\section{Neighbourhood of the Petersen Graph}\label{sec:petersen}

The uniform blowup of the Petersen graph \(P\), is an extremal example for Conjecture~\ref{mainconjecture}, that is every set of \(\lfloor n/2 \rfloor\) vertices spans at least \(n^2/50\) edges. Here we show that the Conjecture~\ref{mainconjecture} holds for any graph that is close to a uniform blowup of Petersen graph  in the sense of Definition~\ref{d:epsClose}. By Theorem~\ref{t:CloseMain} it is enough to show that sufficiently balanced blowups of $P^{*}$ (see Figure~\ref{f:P15}) have uniform sparse halves.

\begin{figure}[htbp]
\begin{center}
\begin{tikzpicture}[label distance=1mm]
    \tikzstyle{every node}=[draw,circle,fill=black,minimum size=15pt,
                            inner sep=0pt]
    \draw (0,0) node (v10){\textcolor{white}{\scriptsize$\boldsymbol{v_{10}}$}}
        [ultra thick]-- ++(0:\vdist) node (v7) {\textcolor{white}{\scriptsize$\boldsymbol{v_7}$}}
        -- ++(216:\vdist) node (v9){\textcolor{white}{\scriptsize$\boldsymbol{v_9}$}}
        -- ++(72:\vdist) node (v6) {\textcolor{white}{\scriptsize$\boldsymbol{v_6}$}}
        -- ++(288:\vdist) node (v8) {\textcolor{white}{\scriptsize$\boldsymbol{v_8}$}}
        -- ++(306:\udist) node (v3){\textcolor{white}{\scriptsize$\boldsymbol{v_3}$}};
	\path (v10) ++(162:\udist) node (v5){\textcolor{white}{\scriptsize$\boldsymbol{v_5}$}};
	\path (v7) ++(18:\udist) node (v2){\textcolor{white}{\scriptsize$\boldsymbol{v_2}$}};
	\path (v9) ++(234:\udist) node (v4){\textcolor{white}{\scriptsize$\boldsymbol{v_4}$}};
	\path (v6) ++(90:\udist) node (v1){\textcolor{white}{\scriptsize$\boldsymbol{v_1}$}};
	
	 [ultra thick]\path (v10) ++(342:0.52573111211*\vdist) node 
	  (invisible)[fill=white,draw=white] {};

  	\path(invisible) ++(126+\shift:\wdist) node (w1){\scriptsize\textcolor{white}{$\boldsymbol{w_1}$}};
	\path(invisible) ++(198+\shift:\wdist) node (w5){\scriptsize\textcolor{white}{$\boldsymbol{w_5}$}};
	\path(invisible) ++(270+\shift:\wdist) node (w4){\scriptsize\textcolor{white}{$\boldsymbol{w_4}$}};
	\path(invisible) ++(342+\shift:\wdist) node (w3){\scriptsize\textcolor{white}{$\boldsymbol{w_3}$}};
	\path(invisible) ++(414+\shift:\wdist) node (w2){\scriptsize\textcolor{white}{$\boldsymbol{w_2}$}};
      \draw [ultra thick] (v10)--(v8);
\draw [ultra thick] (v10)--(v5);
\draw  [ultra thick](v7)--(v2);
\draw [ultra thick] (v9)--(v4);
\draw  [ultra thick](v6)--(v1);
\draw [ultra thick] (v8)--(v3);
\draw [ultra thick] (v5)--(v1);
\draw [ultra thick] (v2)--(v1);
\draw [ultra thick] (v2)--(v3);
\draw [ultra thick] (v4)--(v3);
\draw [ultra thick] (v4)--(v5);

     \draw (w1)--(v1);
     \draw (w1)--(v4);
     \draw (w1)--(v7);
     \draw (w1)--(v8);

     \draw (w2)--(v2);
\draw (w2)--(v5);
\draw (w2)--(v8);
\draw (w2)--(v9);

\draw (w3)--(v1);
\draw (w3)--(v3);
\draw (w3)--(v10);
\draw (w3)--(v9);

\draw (w4)--(v2);
\draw (w4)--(v4);
\draw (w4)--(v6);
\draw (w4)--(v10);

\draw (w5)--(v5);
\draw (w5)--(v3);
\draw (w5)--(v6);
\draw (w5)--(v7);

\end{tikzpicture}

\caption{The graph $P^{*}$.\label{f:P15}}
\end{center}

\end{figure}
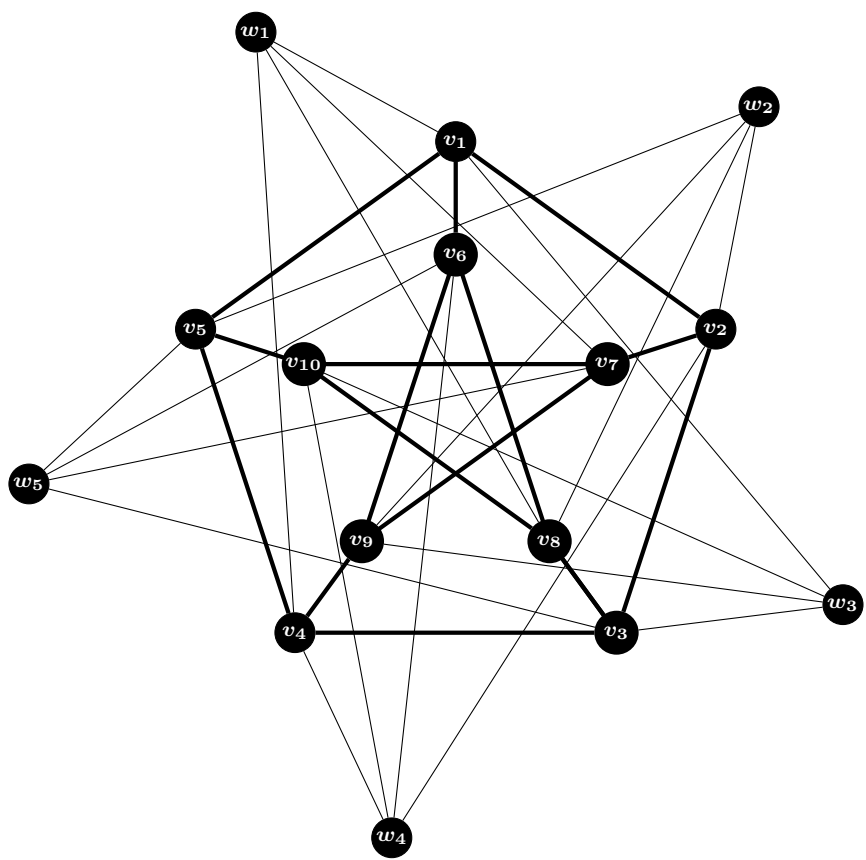

\begin{lem}\label{l:PUniform} Let the weighted graph $(P^{*},\omega)$ be $(1/500)$-balanced. Then it has a \(\frac{1}{80}\)-uniform sparse half.
\end{lem}
\begin{proof} Let $\delta:=1/500$ and the vertices of $P^{*}$ be labeled as in Figure~\ref{f:P15}, where \(V:=V(P) = \{v_1, v_2, \dots, v_{10}\}\) and \(W:=V(P^*) \setminus V(P)=\{w_1, w_2, \dots, w_5\}\). We define a collection \(\{\pl{s}_{i,j}\}_{i \in [5],j \in [4]}\) of halves of $(P^{*},\omega)$. Fix \(i \in [5]\) and consider the vertex \(w_i\). Let \(M_i := V \setminus N(w_i)\) and not that $M_i$ induces a matching of size three. Choose vertices \(\{v_{i_1}, v_{i_2}, v_{i_3}\}\subseteq V(M_i)\) such that they are independent and there exist unique \(w_{i_j} \neq w_i\) such that  every \(v\in  V(M_i)\setminus \{v_{i_1}, v_{i_2}, v_{i_3}\}\) is adjacent to \(w_{i_j}\).  Note that for every \(i\) there exist four such choices of \(\{v_{i_1}, v_{i_2}, v_{i_3}\}\), fix one of them and assign
\[\pl{s}_{i,j}(w_q)  = \begin{cases} \omega(w_q), &\mbox{ if } q=i, \\
\frac{1}{4}\omega(w_q), & \mbox{ if } q=i_j, \\
0,&\mbox{ otherwise}. \end{cases}\]
and
\[\pl{s}_{i,j}(v_{k})  = \begin{cases} 0, &\mbox{ if } v_k\notin V(M_i) \\
\omega(v_{k}), &\mbox{ if } k=j_1,j_2,j_3, \\
\frac{1}{3}\left(\frac{1}{2}-(\omega(v_{i_1}) + \omega(v_{i_2}) + \omega (v_{i_3}) + \omega (w_i) + \frac{1}{4}\omega(w_{i_j}))\right), & \mbox{ otherwise}. \end{cases}\]
It is easy to check that every \(\pl{s}_{i,j}\) is a half. Let \(\textbf{s}\) be a distribution concentrated on \(\{\pl{s}_{i,j}\}_{i \in [5],j \in [4]}\) with each of the halves having the same probability $1/20$. We show that \(\textbf{s}\) is a $(1/80)$-uniform sparse half for $\omega$.

First, we show that $\E{\textbf{s}(e)} \geq \frac{1}{80} \cdot {\omega(e)}$ for every edge \(e\in E(P^{*})\). It can be routinely checked that for our choice of $\delta$ one has
\begin{equation}\label{e:petersen1}
\frac{1}{3}\left(\frac{1}{2}-(\omega(v_{i_1}) + \omega(v_{i_2}) + \omega (v_{i_3}) + \omega (w_i) + \frac{1}{4}\omega(w_{i_j}))\right)\geq \frac{1}{3} \omega(v_k),
\end{equation}
for every \(v_k \in V(M_i) \setminus\{v_{i_1}, v_{i_2}, v_{i_3}\}\).
If both ends $v,v'$ of $e \in E(P^{*})$ lie in $V$ then, using (\ref{e:petersen1}), we have
\[\E{\textbf{s}(e)} \geq \frac{4}{20}\cdot \frac{1}{3}\omega(v)\omega(v') = \frac{1}{15} \omega(e) \geq \frac{1}{80}{\omega(e)}.\]
If $e$ joins $v \in V$ and $w \in W$ then
\[\E{\textbf{s}(e)} \geq \frac{3}{20} \cdot \frac{1}{4}\omega(w_i) \cdot \frac{1}{3} \omega(v_j) = \frac{1}{80}\omega(e) \]

It remains to prove that \(\E{\textbf{s}(E(G))} \leq  \frac{1}{50}.\) Note that,
\begin{align*}
\pl{s}_{i,j}(E(G)) &= \left(\frac{1}{2} -(\omega(v_{i_1}) + \omega(v_{i_2}) + \omega (v_{i_3}) + \omega (w_i) +\frac{1}{4}\omega(w_{i_j}))\right) \times\\ &\times \left(\frac{1}{4}\omega(w_{i_j}) + \frac{1}{3}\left(\omega(v_{i_1}) + \omega(v_{i_2}) + \omega (v_{i_3})\right)\right).
\end{align*}

We finish the proof using the following technical lemma, the  proof of which is included in the appendix.
\begin{lem}
\label{petersentechnicallem}
Suppose given are \(x_1, x_2, \dots , x_{10}, y_1, y_2, \dots y_5\) reals and \[L(y_i)=\{x_{i+1}, x_{i+2}, x_{i+4}, x_{i+5},x_{i+8}, x_{i+9}\},\] for each \(1\leq i \leq 5\) such that \(0\leq x_i\leq 1\) , \(0\leq y_j \leq 1\) and \(\sum_{i=1}^{10}{x_i} +\sum_{j=1}^{5}{y_j}= 1\). If there exists some \(0<\delta \leq \frac{1}{90}\) such that \(x_i\geq \frac{1}{10}-\delta\)  for each \(i=1,2,\dots,10\) then
 \begin{align}
 \label{petersensparsehalf}
 \sum_{i\neq j \atop{x_{{i,j}_1}, x_{{i,j}_2},x_{{i,j}_3} \atop{\in L(y_i) \cap L(y_j)}}}\left(\frac{1}{2} -x_{{i,j}_1} -x_{{i,j}_2} -x_{{i,j}_3} - y_i-\frac{1}{4}y_j\right)  \left(\frac{1}{4}y_j\ + \frac{1}{3}\left(x_{{i,j}_1} + x_{{i,j}_2} + x_{{i,j}_3}\right)\right)\leq\frac{2}{5}.
\end{align}
\end{lem}
It is easy to see that \(\E{\textbf{s} (E(G))} \leq  \frac{1}{50}.\) follows from Lemma~\ref{petersentechnicallem} applied with \(x_i := \omega(v_i)\) and \(y_j :=\omega(w_j)\). Thus $\textbf{s}$ is a $1/80$-uniform sparse half, as claimed.
\end{proof}

As promised, Lemma~\ref{l:PUniform} implies the main theorem of this section.

\begin{thm}
\label{petersenmaintheorem} There exists \(\delta>0\) such that any triangle-free graph \(G\) on \(n\) vertices which can be \(\delta\)-approximated by the Petersen graph has a sparse half.
\end{thm}

\begin{proof}The theorem follows from Theorem~\ref{t:CloseMain}, as the Petersen graph satisfies the requirements of that theorem with $\alpha=1/500$ by Lemma~\ref{l:PUniform}.
\end{proof}

\bibliographystyle{amsplain}
\bibliography{lib}

\providecommand{\bysame}{\leavevmode\hbox to3em{\hrulefill}\thinspace}
\providecommand{\MR}{\relax\ifhmode\unskip\space\fi MR }
\providecommand{\MRhref}[2]{%
  \href{http://www.ams.org/mathscinet-getitem?mr=#1}{#2}
}
\providecommand{\href}[2]{#2}
\begin{thebibliography}{10}

\bibitem{erdossos}
B.~Andr\'{a}sfai, P.~Erd\H{o}s, and V.~T. Sos, \emph{{ On the connection
  between chromatic number, maximal clique and minimal degree of a graph}},
  Discrete Math. (1974), no.~8, 205--218.

\bibitem{jinchenkoh}
C.C. Chen, G.P. Jin, and K.M. Koh, \emph{Triangle-free graphs with large
  degree}, Combinatorics, Probability and Computing (1997), no.~6, 381--396.

\bibitem{Erdosfirst}
P.~Erd\H{o}s, \emph{Problems and results in graph theory and combinatorial
  analysis}, Proceedings of the Fifth British Combinatorial Conference (1975),
  no.~15(1976), 169--192.

\bibitem{Erdos3}
P.~Erd\H{o}s, \emph{Some old and new problems in various branches of
  combinatorics}, Graphs and Combinatorics, Discrete Math (1995,1997), 165/166
  , 227--231.

\bibitem{Erdos1}
P.~Erd\H{o}s, R.~Faudree, C.~Rousseau, and R.~Schelp, \emph{A local density
  condition for triangles}, Discrete Math (1994), no.~127, 153--161.

\bibitem{haggkvist}
R.~H\"{a}ggkvist, \emph{Odd cycles of specified length in non-bipartite
  graphs}, Annals of Discrete Mathematics \textbf{13} (1982), 89--100.

\bibitem{jin}
G.P. Jin, \emph{Triangle-free chromatic grpahs}, Discrete Mathematics (1995),
  no.~145, 151--170.

\bibitem{Sudakov}
P.~Keevash and B.~Sudakov, \emph{Sparse halves in triangle-free graphs},
  Journal of Combinatorial Theory (2006), no.~96, 614--620.

\bibitem{krivelevich}
M.~Krivelevich, \emph{On the edge distribution in triangle-free graphs}, J.
  Comb. Theory, Ser. B \textbf{63} (1995), no.~2, 245--260.

\bibitem{LovSidorenko}
L.~Lov\'{a}sz, \emph{{Subgraph densities in signed graphons and the local
  Simonovits-Sidorenko conjecture}}, The Electronic Journal of Combinatorics
  [electronic only] \textbf{18} (2011), no.~1, Research Paper P127, 21 p.,
  electronic only (eng).

\bibitem{lovasz}
L.~Lov\'{a}sz and B.~Szegedy, \emph{Testing properties of graphs and
  functions}, Israel Journal of Mathematics (2010), no.~178, 113--156.

\bibitem{RazborovCH}
A.~A. Razborov, \emph{{On the Caccetta-H\"{a}ggkvist Conjecture with Forbidden
  Subgraphs}}, J. Graph Theory \textbf{74} (2013), 236--248.

\bibitem{RazborovTuran34}
A.A. Razborov, \emph{{On Tur\'{a}n's (3,4)-problem with forbidden subgraphs}},
  Mathematical Notes \textbf{95} (2014), no.~1-2, 245--252.

\bibitem{woodall}
D.R. Woodall, \emph{{The binding number of a graph and its Anderson number}},
  J. Combin. Theory (1973), 225--255.

\end{thebibliography}

\section{Appendix}

\subsection*{Proof of Lemma~\ref{8cycletechnicallem}}
Suppose the Lemma is false. Then for each \(i =1,2,\dots, 8\)
\begin{equation}
\label{equation8}
\frac{1}{2}\left(\frac{1}{2}-x_i-x_{i+1}-x_{i+2}\right)\left(x_i +x_{i+2}\right) + \frac{1}{4}\left(\frac{1}{2}-x_i-x_{i+1}-x_{i+2}\right)^2 > \frac{1}{50}.
\end{equation}
Summing up these inequalities over all \(i= 1, 2,\dots , 8\) we get that

\begin{align}
\begin{split}
\label{8thcycle}
\frac{8}{50}&<\sum_{i=1}^{8}{\frac{1}{2}\left(\frac{1}{2}-x_i-x_{i+1}-x_{i+2}\right)\left(x_i +x_{i+2}\right)} + \sum_{i=1}^{8}{\frac{1}{4}\left(\frac{1}{2}-x_i-x_{i+1}-x_{i+2}\right)^2} \\
&=\frac{1}{2} \sum_{i=1}^{8}{x_i} - \sum_{i=1}^{8}{\left(x_i + x_{i+1} + x_{i+2}\right)x_i} + \frac{1}{4}\cdot\frac{1}{4}\cdot8 -\frac{1}{4}\sum_{i=1}^{8}{(x_i+ x_{i+1} + x_{i+2})} \\ &\qquad \qquad \qquad \qquad+ \frac{1}{4} \sum_{i=1}^{8}{(x_i+x_{i+1} + x_{i+2})^2}\\
&=\frac{1}{2}  - \sum_{i=1}^{8}{(x_i + x_{i+1} + x_{i+2})x_i} + \frac{1}{2} -\frac{3}{4} + \frac{1}{4} \sum_{i=1}^{8}{(x_i+x_{i+1} + x_{i+2})^2}\\
&=\frac{1}{4}  - \frac{1}{4}\sum_{i=1}^{8}{{x_i}^2}-\frac{1}{2}\sum_{i=1}^{8}{x_ix_{i+2}}.
\end{split}
\end{align}

Let us find the maximum value under the conditions of the lemma. To find the maximum value of the expression in (\ref{8thcycle}), we need to find the minimum value of
\[S:=\frac{1}{4}\sum_{i=1}^{8}{{x_i}^2}+\frac{1}{2}\sum_{i=1}^{8}{x_ix_{i+2}}.\]

\begin{claim}\label{c:71} For every \(1\leq i \leq 8\)
\[x_{i+1}+ x_{i+2} + x_{i+3} < 0.394\]
\end{claim}

\begin{proof}
By inequality~(\ref{equation8})
\[\frac{1}{2}\left(\frac{1}{2}-(x_{i+1}+x_{i+2}+x_{i+3})\right)\left(x_{i+1} +x_{i+3}\right) + \frac{1}{4}\left(\frac{1}{2}-(x_{i+1}+x_{i+2}+x_{i+3})\right)^2 > \frac{1}{50}.\]
Let \(\alpha = x_{i+1} + x_{i+2} + x_{i+3}\).  Then
$x_{i+1} + x_{i+3} = \alpha - x_{i+2} \leq \alpha - \frac{	1}{14}.$ Therefore
\[
\frac{1}{2}\left(\frac{1}{2}-\alpha\right)\left(\alpha - \frac{1}{14}\right) + \frac{1}{4}\left(\frac{1}{2}-\alpha\right)^2=
-\frac{{\alpha}^2}{4} + \frac{\alpha}{28} + \frac{5}{112} >
 \frac{1}{50},
\]
which reduces to  the inequality
\[\frac{{\alpha}^2}{4} -\frac{\alpha}{28} -\frac{69}{2800} <0. \]
This quadratic inequality gives us the desired \(\alpha < 0.393412\) bound. \end{proof}

It is easy to check that the following claim is true.
\begin{claim}\label{c:72}
\[S \geq
\frac{1}{2}\sum_{i=1}^{4}{{z_i}^2}+\sum_{i=1}^{4}{z_iz_{i+2}},
\]
where  \(z_i = (x_i + x_{i+4})/2\) for all \(1\leq i \leq 4\).
\end{claim}

\begin{claim}\label{c:73} For every \(1\leq i \leq 4\),
\(z_i > 0.106\).
\end{claim}
\begin{proof} For every \(1\leq i \leq 4\)  we have
\begin{align*}
 1 &=
(x_i+ x_{i+4}) + (x_{i+1} + x_{i+2} + x_{i+3}) + ( x_{i+5} + x_{i+6} + x_{i+7} ) \stackrel{(\ref{c:71})}< 2z_i + 2 \cdot 0.394,
\end{align*}
therefore
\(z_i > 0.106\).\end{proof}

Let  \(\beta = z_1 +z_3\). Then
\begin{align*}
S &\stackrel{(\ref{c:72})}{\geq}
\frac{1}{2} {(z_1 + z_3)}^2 +\frac{1}{2}  {(z_2 + z_4)}^2 + z_1z_3 + z_2z_4 \\ &=\frac{1}{2}{\beta}^2 +\frac{1}{2}{\left(\frac{1}{2} - \beta\right)}^2  +  z_1z_3 + z_2z_4   \\
&\stackrel{(\ref{c:73})}{\geq}\frac{1}{2}{\beta}^2 +\frac{1}{2}{\left(\frac{1}{2} - \beta\right)}^2  +  0.106 \cdot (\beta - 0.106)  + 0.106\left(\frac{1}{2}-\beta - 0.106\right) \\
&={\beta}^2- \frac{1}{2}\beta + 0.155528\\
& >0.093
\end{align*}
The last two inequalities hold  because
for fixed \(\beta\) the expression $z_1z_3 +z_2z_4$
achieves its minimum value when \(z_1 = z_2 = 0.106\), \(z_3 =\beta - 0.106\) and  \(z_4 = \frac{1}{2} -\beta - 0.106\). The expression
${\beta}^2- \frac{1}{2}\beta + 0.155528$ achieves its minimum for \(\beta = \frac{1}{4}.\)
It follows from the inequality above that
\[\frac{1}{4}  - \frac{1}{4}\sum_{i=1}^{8}{{x_i}^2}-\frac{1}{2}\sum_{i=1}^{8}{x_ix_{i+2}} = \frac{1}{4} - S  < \frac{8}{50},\]
a contradiction that finishes the proof of the lemma.

\subsection*{Proof of Lemma \ref{11cycletechnicallem}}
Suppose the lemma is false. Then for all \(1 \leq i \leq 11\) we have
\begin{align*}
&\frac{1}{2}\left(\frac{1}{2} - (x_{i+1} +x_{i+2}+x_{i+3} + x_{i+4})\right)(x_{i+1} + x_{i+4}) \\
&+ \frac{1}{4}\left(\frac{1}{2} - (x_{i+1} +x_{i+2}+x_{i+3} + x_{i+4})\right)^2 >1/50.
\end{align*}
Summing these inequalities for \(1 \leq i \leq 11\) we obtain
\begin{align}
\begin{split}
\label{11thcycle}
\frac{11}{50} &<  \sum_{i=1}^{11}{\frac{1}{2}\left(\frac{1}{2} - (x_{i+1} +x_{i+2}+x_{i+3} + x_{i+4})\right)(x_{i+1}+ x_{i+4})} \\
&+\sum_{i=1}^{11}{ \frac{1}{4}\left(\frac{1}{2} - (x_{i+1} +x_{i+2}+x_{i+3} + x_{i+4})\right)^2}\\
&=  \frac{1}{4}\sum_{i=1}^{11}{(x_{i+1} + x_{i+4})} - \sum_{i=1}^{11}{(x_{i+1} +x_{i+2}+x_{i+3} + x_{i+4})x_{i+1}} \\
& + \frac{11}{16}- \frac{1}{4}\sum_{i=1}^{11}{(x_{i+1}+x_{i+2}+x_{i+3} + x_{i+4})} + \frac{1}{4}\sum_{i=1}^{11}{{\left(x_{i+1}+x_{i+2}+x_{i+3} + x_{i+4}\right)}^2} \\
&= \frac{3}{16} - \sum_{i=1}^{11}{(x_{i+1} +x_{i+2}+x_{i+3}+x_{i+4})x_{i+1}} + \frac{1}{4}\sum_{i=1}^{11}{{\left(x_{i+1}+x_{i+2}+x_{i+3} + x_{i+4}\right)}^2} \\
&=\frac{3}{16} + \frac{1}{2} \sum_{i=1}^{11}x_i(x_{i+1}-x_{i+3}) \\
&\leq \frac{3}{16} + \frac{1}{2} \left( \sum_{i=1}^{11}{x_ix_{i+1}} -   \frac{11}{196}\right).
\end{split}
\end{align}
We claim that
\[f(x_1,\ldots,x_{11}):=\sum_{i=1}^{11}{x_ix_{i+1}} \leq \frac{77}{784}.\]
Indeed, consider any  pair \((x_i,x_j)\),  such that  \(j \neq i \pm 1\),  let \(\alpha = x_i + x_j\) is fixed. Note that $f$
is linear as a function  of \((x_i,x_j)\). Therefore \(f(x_i,x_j)\) achieves its maximum value on the region $R:=\{0\leq x_i \leq 1/14$~for~$1 \leq 1 \leq 14$,~$\sum_{i=1}^{11}x_i=1\}$, when \(x_i = \frac{1}{14}\) and \(x_j = \alpha - \frac{1}{14}\), or when \(x_j = \frac{1}{14}\) and \(x_i = \alpha - \frac{1}{14}\). Thus $f$ attains its maximum on $R$ when all  variables are equal \(\frac{1}{14}\) except possibly  two of them whose indices are consecutive, without loss of generality, say \(x_{10}\) and  \(x_{11}\). It is easy to see that the maximum is achieved for $x_{10} = x_{11} = \frac{5}{28}$ and is equal to $\frac{77}{784},$ as claimed.

Thus (\ref{11thcycle}) implies
\[ \frac{11}{50} < \frac{3}{16} + \frac{1}{2} \left( \sum_{i=1}^{11}{x_ix_{i+1}} -   \frac{11}{196}\right) \leq \frac{3}{16} + \frac{1}{2}\cdot \left(\frac{77}{784} - \frac{11}{196}\right)  = \frac{327}{1568} < \frac{11}{50},
\]
a contradiction that finishes the proof.

\subsection*{Proof of Lemma~\ref{petersentechnicallem}} Let \(Y:=\sum_{i=1}^5{y_i}\). Summing inequalities~(\ref{petersensparsehalf}) over all \(i, j\) such that \(i\neq j\) and \(x_{{i,j}_1}, x_{{i,j}_2},x_{{i,j}_3} \in L(y_i) \cap L(y_j)\). We get 
\begin{align*}
&\sum_{i\neq j \atop {}}{\left(\frac{1}{2} -(x_{{i,j}_1} +x_{{i,j}_2} +x_{{i,j}_3} + y_i+\frac{1}{4}y_j)\right)  \left(\frac{1}{4}y_j\ + \frac{1}{3}\left(x_{{i,j}_1} + x_{{i,j}_2} + x_{{i,j}_3}\right)\right)}\\
&=\frac{1}{8}\sum_{i\neq j}{y_j} -\frac{1}{4}\sum_{i\neq j}{y_j\left(x_{{i,j}_1} +x_{{i,j}_2} +x_{{i,j}_3}\right)} +\frac{1}{6}\sum_{i\neq j}(x_{{i,j}_1} +x_{{i,j}_2} +x_{{i,j}_3}) \\
&-\frac{1}{3} \sum_{i\neq j}{\left(x_{{i,j}_1} +x_{{i,j}_2} +x_{{i,j}_3}\right)}^2 -\frac{1}{3} \sum_{i\neq j}\left(y_i+ \frac{1}{4}y_j\right) \left(x_{{i,j}_1} +x_{{i,j}_2} +x_{{i,j}_3}\right) -\frac{1}{4}\sum_{i\neq j}{y_j\left(y_i + \frac{1}{4}y_j\right)}  \\
&\leq\frac{1}{2}Y -3Y\left(\frac{1}{10}-\delta\right) + (1- Y) -\frac{36{(1-Y)}^2}{60} -5Y\left(\frac{1}{10}-\delta\right) \\
&= \frac{2}{5} -\frac{Y}{10} -\frac{3}{5}{Y}^2 + 8\delta Y\\
&\leq \frac{2}{5},
\end{align*}
since \(\delta \leq \frac{1}{90}\).
\end{document}